\documentclass{amsart}
\usepackage[utf8]{inputenc}
\usepackage{amsmath, amsthm, amscd, amssymb, amsfonts, latexsym,amsrefs}
\usepackage{fullpage}
\usepackage{bm}
\usepackage[all]{xy}
\usepackage{tikz}
\usepackage{fancybox}
\usepackage{hyperref}

\usepackage{enumerate}

\newcommand{\kommentar}[1]{}
\newcommand{\ord}{\text{ord}}

\newcommand{\F}{\mathbb F}

\newcommand{\Z}{\mathbb Z}

\newcommand{\R}{\mathbb R}

\newcommand{\C}{\mathbb C}

\DeclareMathOperator{\Res}{Res}

\DeclareMathOperator{\re}{Re}

\renewcommand{\pmod}[1]{\,(\mathrm{mod}\,#1)}

\newtheorem{lem}{Lemma}[section]
\newtheorem{prop}[lem]{Proposition}
\newtheorem{thm}[lem]{Theorem}
\newtheorem{defn}[lem]{Definition}
\newtheorem{cor}[lem]{Corollary}

\theoremstyle{definition}

\newtheorem{rem}[lem]{Remark}

\title{On the Northcott property of zeta functions over function fields}
\author{Xavier Généreux, Matilde Lal\'in, Wanlin Li}
\date{}

\address{Xavier G\'en\'ereux:  D\'epartement de math\'ematiques et de statistique,
                                    Universit\'e de Montr\'eal.
                                    CP 6128, succ. Centre-ville.
                                     Montreal, QC H3C 3J7, Canada}\email{xavier.genereux@umontreal.ca}

\address{Matilde Lal\'in:  D\'epartement de math\'ematiques et de statistique,
                                    Universit\'e de Montr\'eal.
                                    CP 6128, succ. Centre-ville.
                                     Montreal, QC H3C 3J7, Canada}\email{mlalin@dms.umontreal.ca}

\address{Wanlin Li: Centre de recherches math\'ematiques,
                                    Universit\'e de Montr\'eal.
                                    CP 6128, succ. Centre-ville.
                                     Montreal, QC H3C 3J7, Canada}\email{liwanlin@crm.umontreal.ca}


\begin{document}

\begin{abstract}
Pazuki and Pengo defined a Northcott property for special values of zeta functions of number fields and certain motivic $L$-functions. We determine the values for which the Northcott property holds over function fields with constant field $\F_q$ outside the critical strip. We then use a case by case approach for some values inside the critical strip, notably $\re (s) < \frac{1}{2} - \frac{\log 2}{\log q}$ and for $s$ real such that $1/2 \leq s \leq 1$, and we obtain a partial result for complex $s$ in the case $1/2< \re(s)\leq  1$ using recent advances on the Shifted Moments Conjecture over function fields. 
\end{abstract}

\maketitle

\section{Introduction}
The Northcott property \cite{Northcott} implies that a set of algebraic numbers with bounded height and bounded degree must be finite. In \cite{PP}, Pazuki and Pengo study a variant of the Northcott property for number fields using special values of the Dedekind zeta function to measure the height.  For a field $K$ and $s\in \C$ denote 
\[\zeta_K^*(s):=\lim_{t\rightarrow s}\frac{\zeta_K(t)}{(t-s)^{\ord_{s}(\zeta_K(t))}},\]
the first nonzero coefficient of the Taylor series for $\zeta_K(s)$ around $s$. 

For a fixed $s=n\in \Z$, Pazuki and Pengo  consider, for $B$ a fixed positive real number, the set of isomorphism classes of number fields \begin{equation}\label{eq:PPset} \{[K]: |\zeta_K^*(n)|\leq B\},\end{equation}
and discuss the finiteness of this set under various conditions of $B$ and $n$. For number fields, they prove that a Northcott property holds for $n$ located at the left of the critical strip, but does not hold for $n$ to the right of the critical strip, and they show that such a property does not hold for $n=1$, but holds for $n=0$. They also estimate the size of this set when the Northcott property holds.

We are interested in exploring the Northcott property for global function fields, more precisely, we consider the set of isomorphism classes of function fields $K$ with constant field $\F_q$, and we aim at considering the value of its zeta function at any complex number $s \in \C$. To this end, we define 
\[S_{q,s,B} = \{ [K] : |\zeta_K^*(s)| \leq B \},\]
where $[K]$ denotes the isomorphism class of $K$, a global function field in one variable over a finite constant field $\mathbb{F}_q$ with $q$ elements, where $q$ is fixed.
The Northcott property of $\F_q$ at  $s$  is equivalent to having $S_{q,s,B}$ finite for all $B\in \mathbb{R}_{>0}$. More generally, we consider the following definition.

\begin{figure}
    \centering
    \begin{tikzpicture}[scale=2]
        \filldraw [blue!20] (1/3,-2) -- (1/3,2) -- (-2,2) -- (-2,-2);
        \filldraw [red!25] (1,-2) -- (1,2) -- (2,2) -- (2,-2);
        \filldraw [red!25] (1/2,-2) -- (1/2,2) -- (1,2) -- (1,-2);
        \draw [dashed] (1,-2) -- (1,2);
        \draw (-2,0) -- (2,0);
        \draw (0,-2) -- (0,2);
        \draw [red!60, ultra thick] (1/2,0) -- (1,0);
        \filldraw [red] (1,0) circle (1pt);
        \filldraw [red] (1/2,0) circle (1pt);
        \draw [dashed] (1/2,-2) -- (1/2,2);
        \node at (1/2,-0.2) {$1/2$};
    \end{tikzpicture}
    \caption{For the base field $\mathbb{F}_{q}$, the Northcott property holds in the area in blue. Bright red indicates that $S_{q,s,B}$ is infinite for all $B$, while light red means that $S_{q,s,B}$ is infinite for $B$ greater than a certain constant. Remark that the white gap corresponding to $\frac{1}{2}-\frac{\log 2}{\log q}\leq \sigma <\frac{1}{2}$ disappears as $q \rightarrow \infty$.}
    \label{fig:covering}
\end{figure}
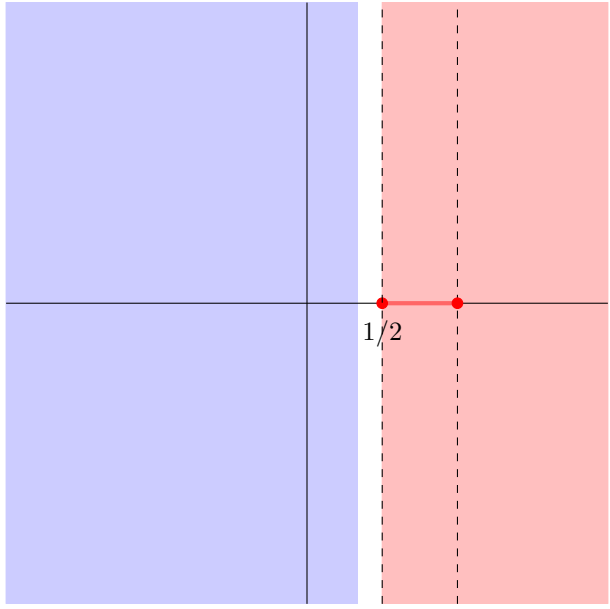

\begin{defn} Let $q$ be a power of a prime, $s \in \C$,  and $\mathcal{I}\subset \R_{\geq 0}$.  We say that the triple $(q,s,\mathcal{I})$ has the Northcott property if the set $S_{q,s,B}$ is finite for all $B \in \mathcal{I}$. We say that $(q,s,\mathcal{I})$ is non-Northcott if the set $S_{q,s,B}$ is infinite for all $B \in \mathcal{I}$.
\end{defn}

We prove the following results. 
\begin{thm}\label{thm:main} Let $\sigma=\re(s)$. 
\begin{enumerate}[$a)$]
\item If $q>4$, then $(q,0,\R_{>0})$ satisfies the Northcott property.

    \item When $\sigma < \frac{1}{2} - \frac{\log 2}{\log q}$, then $(q,s,\R_{>0})$ satisfies the Northcott property.

    \item Let $\sigma>1$ and 
    \[B=\frac{1}{(1-q^{-\sigma})(1-q^{1-\sigma})^2}.\]
    Then $(q,s,\R_{\geq B})$ is non-Northcott.

    \item For $q \equiv 1 \pmod{4}$, $(q,1,\R_{>0})$ is non-Northcott.

    \item  For $q \equiv 1 \pmod{4}$ and  $1/2<\sigma < 1$, $(q,\sigma,\R_{>0})$ is non-Northcott.  
    
    \item $(q,1/2,\R_{>0})$ is non-Northcott  by working directly with $\zeta_K(1/2)$ (as opposed to $\zeta_K^*(1/2)$).
    
    \item  For $q \equiv 1 \pmod{4}$ and  $1/2<\sigma$ (but $s\not = 1$), let 
    \begin{align*}
    B = &\left|\frac{1}{\left(1-q^{-s}\right)\left(1-q^{1-s}\right) }\right|\\&\times  \prod_{\substack{P\, \text{monic}\\ \text{irreducible}}} \left[\frac{1}{2}\left( \left(1-\frac{1}{|P|^{s-\frac{1}{2}}}\right)^{-1}\left(1-\frac{1}{|P|^{\overline{s}-\frac{1}{2}}}\right)^{-1}+\left(1+\frac{1}{|P|^{s-\frac{1}{2}}}\right)^{-1}\left(1+\frac{1}{|P|^{\overline{s}-\frac{1}{2}}}\right)^{-1}\right)+\frac{1}{|P|}\right]^{1/2}\\&\times \left(1+\frac{1}{|P|}\right)^{-1/2}.
    \end{align*}
Then $(q, s,\R_{>B})$ is non-Northcott. 
\end{enumerate}
\end{thm}

In the list of Theorem \ref{thm:main}, statements $d)$, $e)$, and $f)$ depend on deep results on the distribution of quadratic Dirichlet $L$-functions due to Lumley \cites{Lumley,Lumley2} and Li \cite{Li}, while $g)$  for $\sigma<1$ depends on the Shifted Moments Conjecture for quadratic Dirichlet $L$-functions, which was formulated by Andrade and Keating \cite{Andrade-Keating-conj} for the function field case, and has been recently proven under certain constraints by Bui, Florea, and Keating \cite{Bui-Florea-Keating}. The result of $g)$ for $\sigma\geq 1$ follows directly from a moment computation and improves upon the set given in $c)$. The result of $f)$ is for $\zeta_K(1/2)$ as it rests on the existence of infinitely many $[K]$ such that $\zeta_K(1/2)=0$, instead of working with $\zeta_K^*(1/2)$.

In addition, we discuss bounds for $\# S_{q,s,B}$ in the cases of $a)$ and $b)$. More precisely, we use a result of Couveignes \cite{Couveignes} to prove that there is an absolute computable constant $\mathcal{Q}$ (independent of $q$) such that 
\begin{equation}\label{eq:Sbound}
\# S_{q,s,B}\leq q^{\mathcal{Q} c_{\sigma}(\log B)^3B},
\end{equation}
where 
\[c_\sigma=\frac{1}{(\log q)\log \left(q^{\frac{1}{2}-\sigma}-1\right)}.\]

Pazuki and Pengo \cite{PP} study the Northcott, Bogomolov, and Lehmer properties for special values of $L$-functions evaluated at $n\in \Z$. They prove that the Northcott property holds at the left of the critical strip for a general family of motivic $L$-functions (assuming meromorphic continuation and functional equation), which can be compared to our result of Theorem \ref{thm:main} $b)$. They also focus on the Northcott property for Dedekind zeta functions of number fields evaluated at integer numbers and prove that it is not satisfied for $n\in \Z_{\geq 1}$, in a result analogous to our  Theorem \ref{thm:main} $c)$. When $n\in \Z_{\leq 0}$, they obtain bounds for the size of the set given in  \eqref{eq:PPset} which are better than \eqref{eq:Sbound} in the case of $n<0$, but worse than  \eqref{eq:Sbound} in the case of $n=0$. Our results are limited by the lack of understanding on the number of smooth, projective curves of genus $g$ over a fixed finite field.

This article is organized as follows. Section \ref{sec:background} covers standard background on the zeta function attached to a global function field. Sections \ref{sec:left} and \ref{sec:right} treat the left and right sides of the critical strip, while Section \ref{sec:critical} considers the critical strip.

\section*{Acknowledgements} The authors are grateful to Jordan Ellenberg, Alexandra Florea, and Allysa Lumley for many useful discussions.  This work was supported by the
Centre de recherches math\'ematiques and the Institut des sciences math\'ematiques (CRM-ISM postdoctoral fellowship to WL, ISM summer internship to XG),  the Natural Sciences and Engineering Research Council of Canada (Discovery Grant \texttt{355412-2013} to ML), and the Fonds de recherche du Qu\'ebec - Nature et technologies (Projet de recherche en \'equipe  \texttt{256442} and \texttt{300951} to ML).

\section{Some background on $\zeta_K(s)$} \label{sec:background}
In this section we recall some background on function fields with constant field $\F_q$. More details can be found in \cite{Rosen}.

Let $K$ be a global function field in one variable with a finite constant field $\F_q$ with $q$ elements. A prime of $K$ is a discrete valuation ring $R$ with maximal ideal $P$ such that $\F_q[T] \subset R$ and the quotient field of $R$ equals $K$. The degree of a prime $P$, denoted by $\deg(P)$ is the (finite) dimension of $R/P$ over $\F_q$.  The group of divisors of $K$ is the free abelian group generated by the primes. Thus, a divisor is an element of the form $A=\sum_P a(P)P$. In this case, $\sum_Pa(P)\deg(P)$ is called the degree of $A$, denoted $\deg(A)$. The norm of $A$ is equal to $q^{\deg(A)}$ and is denoted by $|A|$. The divisor $A$ is said to be effective if $a(P)\geq 0$ for all $P$. We write $A\geq 0$ to indicate that $A$ is effective. 

The zeta function of $K$ is defined for $\re(s)>1$ by 
\[\zeta_K(s):=\sum_{A\geq 0} \frac{1}{|A|^s}=\sum_{n=0}^\infty \frac{b_n}{q^{ns}},\]
where the sum is taken over all the effective divisors, and $b_n$ is the number of effective divisors of degree $n$.  Notice that $\zeta_K(s)$ satisfies an Euler product
\[\zeta_K(s)=\prod_{P}\left(1-\frac{1}{|P|^s}\right)^{-1}.\]

From this, we have
\begin{equation}\label{eq:logzeta}
\log \zeta_K(s) =\sum_{P}\sum_{j=1}^\infty \frac{1}{j|P|^{js}}=\sum_{A\geq 0} \frac{\Lambda(A)}{\deg(A)|A|^s},
\end{equation}
where $\Lambda(A)$ is the von Mangoldt function, equal to $\deg(P)$ if $A=P^j$ (or $A=jP$ if written additively) for $P$ prime and $0$ otherwise.

By the Weil conjectures (\cite{Rosen}*{Theorem 5.9}), there is a polynomial $L_K(u) \in \Z[u]$ of degree $2g$, where $g$ is the genus of the curve whose function field is $K$, such that 
\begin{equation}\label{eq:zetadefn}\zeta_K(s)=\frac{L_K(q^{-s})}{(1-q^{-s})(1-q^{1-s})}.\end{equation}
The right hand side provides a meromorphic continuation for $s \in \C$. We immediately see that  $\zeta_K(s)$ has simple poles at $s=0,1$. If we set 
\[\xi_K(s)=q^{(g-1)s}\zeta_K(s),\]
then we have the functional equation
\begin{equation}\label{functeq}
\xi_K(1-s)=\xi_K(s).
\end{equation}

The Riemann Hypothesis, which is known to be true in this context, implies that the zeros of $\zeta_K(s)$ occur only at $\re(s)=1/2$. 

The function $\zeta_K(s)$ admits certain symmetry inherited from the functional equation. This symmetry centers on the critical line $\re(s)= 1/2$. It is natural to analyze the behavior of $|\zeta_K(s)|$ depending on the position of $s$ respect to the critical strip. 

After making the change of variables $u=q^{-s}$, we can write 
\begin{equation}\label{eq:L-factor}
L_K(u)=\prod_{j=1}^{2g} (1-\pi_j u),
\end{equation}
where $|\pi_j|=\sqrt{q}$. In addition, it is known that $L_K(0)=1$ and $L_K(1)=h_K$, the class number of $K$.
By the functional equation, the $\pi_j$'s can be separated in pairs of complex conjugates, so that $\pi_j=\overline{\pi_{2g-j}}$. Notice that we also have (\cite{Rosen}*{Theorem 5.12})
\begin{equation}\label{eq:sumofpi}\sum_{j=1}^{2g} \pi_j^\ell =q^\ell+1-\sum_{d\mid \ell}da_d,\end{equation}
where
\begin{equation}\label{eq:a_ddefn}
a_d=\#\{P:\deg(P)=d\}.
\end{equation}
Thus, we can write the Euler product for $\zeta_K(s)$ as 
\begin{equation}\label{eq:Eulera}
\zeta_K(s)=\prod_{n=1}^\infty \left(1-\frac{1}{q^{ns}}\right)^{-a_n}.
\end{equation}

In particular for $K=\F_q(T)$ we have 
\[\zeta_{\F_q(T)}(s)=\prod_P\left(1-\frac{1}{|P|}\right)^{-1},\]
where the product is over all the primes of $\F_q(T)$, namely, the monic irreducible polynomials and the prime at infinity. When $P$ is a monic irreducible polynomial, we have that $|P|=q^{\deg(P)}$, while the prime at infinity has norm $q$, since its degree is 1. Notice that 
\[\zeta_{\F_q(T)}(s)=\frac{1}{(1-q^{-s})(1-q^{1-s})}.\]
We will also denote by $\zeta_{\F_q[T]}$ the zeta function without the prime at infinity. In this case 
\[\zeta_{\F_q[T]}(s)=\frac{1}{1-q^{1-s}}.\]

Throughout this paper we will write $s=\sigma+i\tau$, where $\sigma, \tau$ are real numbers.

\section{The left side of the critical strip}\label{sec:left}

 Starting with the left side of the critical strip, we obtain a positive result for a large subset of $\mathbb{C}_{\sigma<1/2}$ that contains $\mathbb{C}_{\sigma\leq 0}$ (for $q>4$). More precisely, we prove in Theorems \ref{thm:s=0} and \ref{thm:s<0} that for $s$ in  the blue area of Figure \ref{fig:covering}, $(q,s,\R_{>0})$ has the Northcott property.
 
\begin{lem}\label{lem:Hassesimilar}
The polynomial $L_K(u) \in \mathbb{Z}[u]$ satisfies  the following bounds
\begin{align}\label{eq:doubleHasse}
\left( \sqrt{q}|u|-1\right)^{2g}\leq |L_K(u)| \leq \left( \sqrt{q}|u|+1\right)^{2g}.\end{align}
\end{lem}
\begin{proof}
This follows immediately from equation \eqref{eq:L-factor} and the triangle inequality on each factor. 
\end{proof}

\begin{lem}\label{lem:finitegenus} Let $q$ be a power of a prime $p$.  For a fixed $g$ there are finitely many isomorphism classes of global function fields over $\F_q$ with genus $g$. 
\end{lem}

\begin{proof}
The statement follows from the fact that there exists a moduli stack $\mathcal{M}_g$ over $\mathbb{F}_p$ classifying smooth proper curves of genus $g \ge 2$. (See for example \cite{DeJongKatz}.)
\end{proof}

We are ready to prove the main result of this section. First we treat the case $s=0$ separately. 
\begin{thm}\label{thm:s=0} Let $q$ be a power of a prime such that $q>4$. 
We have that $(q,0,\R_{>0})$ satisfies the Northcott property. 
\end{thm}
\begin{proof}
We remark that 
\[\zeta_K^*(0)=\lim_{s\rightarrow 0} \frac{sL_K(q^{-s})}{(1-q^{-s})(1-q^{1-s})}=\frac{h_K}{1-q}\lim_{s\rightarrow 0}\frac{s}{1-q^{-s}}=\frac{h_K}{(1-q)\log q}.\]

By Lemma \ref{lem:Hassesimilar}, we have that 
\begin{align}\label{eq:Hasse}\frac{(\sqrt{q}-1)^{2g}}{(q-1)\log q} \leq |\zeta_K^*(0)|.\end{align}
Thus, we conclude that $|\zeta_K^*(0)| \to \infty$ as long as $g \rightarrow \infty$. Therefore, if $[K]\in S_{q,0,B}$, we must have that $g(K)$ is bounded.  By Lemma \ref{lem:finitegenus}, there are only finitely many $[K]$ for each $g$, and we conclude that $S_{q,0,B}$ must be finite. 
\end{proof}

\begin{thm}\label{thm:s<0}
Let  $s=\sigma+i \tau\in \C^*$ such that 
\[\sigma<1/2- \frac{\log 2}{\log q},\]
then $(q,s,\R_{>0})$ satisfies the Northcott property. 
\end{thm}

\begin{proof} First notice that if 
$\sigma<1/2- \frac{\log 2}{\log q}$, then we have that $2 < q^{\frac{1}{2}-\sigma}$. From here, we deduce that 
\[\sqrt{q}|u|-1>1.\]
By Lemma \ref{lem:Hassesimilar}, we conclude that 
 $|L_K(u)|\rightarrow \infty$ as long as $g \rightarrow \infty$. We reach the same conclusion as long as $s\not = 0$, since the denominator in $\zeta_K(s)$ is non-zero and bounding $|\zeta_K(s)|$ is therefore equivalent to bounding $|L_K(u)|$.  Thus, if $[K]\in S_{q,s,B}$, we must have that $g(K)$ is bounded.  By Lemma \ref{lem:finitegenus}, there are only finitely many $[K]$ for each $g$, and we conclude that $S_{q,s,B}$ must be finite.

\end{proof}

A natural question is to bound the size of $S_{q,s,B}$ in the cases when it is finite. We will need the following result of Couveignes \cite{Couveignes}.

\begin{thm}\cite{Couveignes}*{Theorem 2 (simplified version)} \label{thm:Couveignes} There exists an absolute and computable constant $\mathcal{Q}$ such that the following is true. Let $K=\F_q(T)(\mathcal{C})=\F_q(T,X)$ be a function field of genus $g\geq 2$ and degree $[\F_q(T,X):\F_q(T)]=n$. Then $K$ is determined by at most 
\[\mathcal{Q}(\log n)^2 (g+n(1+\log_q n))\]
 parameters of $\F_q$. 
 \end{thm}
 Although Couveignes does not give the value of $\mathcal{Q}$, this constant is only related to the technicalities of the proof  and is independent of the base field. In our context this means that it is independent of $q$.
 
 Using Theorem \ref{thm:Couveignes}, we can prove the following bound. 
 \begin{thm}\label{thm:bound} Let $\varepsilon>0$ and $s\in \C$ such that  $\sigma<1/2-\frac{\log 2}{\log q}$.  
Then, as $B\rightarrow \infty$, we have
\[\# S_{q,s,B}\leq q^{\mathcal{Q} c_\sigma (\log B)^3B},\]
where 
\[c_\sigma=\frac{1}{(\log q)\log \left(q^{\frac{1}{2}-\sigma}-1\right)}.\]
\end{thm}
 
 \begin{proof}
 By Theorem \ref{thm:Couveignes}, the number of possible fields $K$ of genus $g$ and degree $n$ is bounded by 
 \[q^{\mathcal{Q} (\log n)^2 (g+n(1+\log_q n))}.\]
 
 We need to count over all possible values of $n$. We can take $n$ as the gonality of the curve $\mathcal{C}$, defined as the smallest possible degree of a dominant map $\mathcal{C}\longrightarrow \mathbb{P}^1(\F_q(T))$, and known to be bounded by $2g-2$ when $g>1$ (see \cite{Poonen}*{Proposition A.1}). For $g=1$, we have an elliptic curve, and we can bound the degree of the function field by 2. Thus, we have a bound for the number of isomorphism classes of fields under consideration with fixed genus $g$ given by  
 \begin{align*}
 \sum_{n=2}^{2g-2}q^{\mathcal{Q} (\log n)^2 (g+n(1+\log_q n))} \leq & q^{\mathcal{Q} (\log (2g-2))^2 (g+(2g-2)(1+\log_q (2g-2))+1} \qquad (g\geq 2).
 \end{align*}
  For $g=1$ this number is bounded by \[1+q^{\mathcal{Q} (\log 2)^2 (1+2(1+\log_q 2))},\] and for $g=0$, we have just one field. 
 
We proceed to let $g$ vary. First consider the case $s=0$ with $q>4$. By equation \eqref{eq:Hasse}, we have
\[\frac{(\sqrt{q}-1)^{2g}}{(q-1)\log q} \leq B,\]
and this gives 
\[g\leq \frac{\log((q-1)(\log q)B)}{2\log(\sqrt{q}-1)}=a_0 \log B+b_0,\]
where $a_0=\frac{1}{2\log\left(\sqrt{q}-1\right)}$ and  $b_0=\frac{\log((q-1)(\log q))}{2\log\left(\sqrt{q}-1\right)}$ denote constants that are only dependent on $q$.

Now consider the case $s\not = 0$ such that $\sigma<1/2-\frac{\log 2}{\log q}$. By equation \eqref{eq:doubleHasse}, 
\[\frac{\left(q^{\frac{1}{2}-\sigma}-1\right)^{2g} }{\left|\left(1-q^{-\sigma-i\tau}\right)\left(1-q^{1-\sigma-i\tau} \right)\right|}\leq B,\]
and this gives 
\[g\leq \frac{\log\left( \left|\left(1-q^{-\sigma-i\tau}\right)\left(1-q^{1-\sigma-i\tau} \right)\right|B\right)}{2\log \left(q^{\frac{1}{2}-\sigma}-1\right)}\leq a_\sigma \log B +b_\sigma,\]
where $a_\sigma=\frac{1}{2\log \left(q^{\frac{1}{2}-\sigma}-1\right)}$ and $b_\sigma=\frac{\log\left(\left(1+q^{-\sigma}\right)\left(1+q^{1-\sigma} \right)\right)}{2\log \left(q^{\frac{1}{2}-\sigma}-1\right)}$ are constants that are only dependent on $\sigma$ and $q$.

Finally, we need to consider the bound summing all the possible values $g$ up to $aB+b$. This gives \begin{align*}
&2+q^{\mathcal{Q} (\log 2)^2 (1+2(1+\log_q 2))} + \sum_{2\leq g \leq aB+b}q^{\mathcal{Q} (\log (2g-2))^2 (g+(2g-2)(1+\log_q (2g-2))+1}\\
&\leq 2+q^{\mathcal{Q} (\log 2)^2 (1+2(1+\log_q 2))}  + q^{\mathcal{Q} (\log (2(aB+b-1)))^2 (aB+b+2(aB+b-1)(1+\log_q (2(aB+b-1)))+2}.
\end{align*}
As $B \rightarrow \infty$ the above is bounded by  
\[\leq q^{\mathcal{Q} \frac{2a(1+o(1))}{\log q} (\log B)^3B}.\]

We conclude by noticing that the formula for $a_0$ is simply the result of specializing $a_\sigma$ at $\sigma=0$.
\end{proof}

\begin{rem}
Lipnowski and Tsimerman \cite{LT}*{Lemma 2.1, Corollary 2.2} estimate the number of possible $L_K(u)$ of fixed genus $g$ to be at most $(2g)^g q^{\frac{1}{4}g(g+1)}$. In the same article,  \cite{LT}*{Eq. (28)} gives a bound for  the number of isomorphism classes on each isogeny class of $p^{\frac{33}{4}g^2(1+o(1))}$. Combining these two estimates gives  a bound of \[(2g)^g q^{\frac{1}{4}g(g+1)}p^{\frac{33}{4}g^2(1+o(1))}\] for the number of isomorphism classes of fields under consideration with fixed genus $g$. While this formula is more explicit than the bound given by Theorem \ref{thm:Couveignes}, the final bound for $\#S_{q,s,B}$ is asymptotically worse as it has $g^2$ in the exponent. In fact, this leads to \[\# S_{q,s,B}\leq q^{\frac{17}{2}a_\sigma^2(1+o(1))B^2}.\]
\end{rem}

\begin{rem}The result of Theorem \ref{thm:bound} can be in principle improved if we use  the argument by de Jong and Katz \cite{DeJongKatz}, which claims that the number of smooth proper curves of genus $g\geq 2$ is bounded by 
\[g^{c_1g}q^{c_2g},\]
where $c_1,c_2$ are (non-effective) positive constants. This leads to 
\[\# S_{q,s,B}\leq q^{c_1a_\sigma(1+o(1))(\log B)B},\]
This bound has a slightly better asymptotic than the result of Theorem \ref{thm:bound}, but it has the disadvantage that we can not compute $c_1$.
\end{rem}

\section{The right side of the critical strip} \label{sec:right}

We now consider the right side of the critical strip, that is, $\C_{\sigma>1}$, where we obtain a  result conditionally on the value of $B$. More precisely, we prove that for the $s$ in the pale shade of red of Figure \ref{fig:covering}, the Northcott property does not hold for $B$ sufficiently large.

The first result allows us to compare $\zeta_K(s)$ with $\zeta_K(\sigma)$.
\begin{lem} \label{lem:sandwichzetas}
Let $s =\sigma + i \tau \in \C$ with $\sigma>1$. Then 
\[\frac{1}{\zeta_K(\sigma)}\leq |\zeta_K(s)|\leq \zeta_K(\sigma).\]
\end{lem}
\begin{proof} We start by proving the upper bound.
Since $\sigma>1$, the Dirichlet series for $\zeta_K(s)$ converges and we can directly bound \begin{align*}
    |\zeta_K(\sigma+i\tau)| = \left| \sum_{n=1}^\infty \frac{b_n}{q^{n(\sigma+i\tau)}}\right| \leq  \sum_{n=1}^\infty \left|\frac{b_n}{q^{n(\sigma+i\tau)}}\right| = \sum_{n=1}^\infty \frac{b_n}{q^{n\sigma}} = \zeta_K(\sigma).
\end{align*}

Now we proceed to prove the lower bound. Notice that for any $\theta$, $1+\cos(\theta)\geq 0$. By considering the logarithm of the Euler product \eqref{eq:logzeta}, we have that 
\[\log\zeta_K(\sigma)+\re \log \zeta_K(\sigma+i\tau) =\sum_{P}\sum_{j=1}^\infty \frac{1+\cos(\tau \log|P^j|)}{j|P|^{j\sigma}}\geq 0.\]
Taking the exponential, we conclude that \begin{equation}\label{eq:ssigma}
\zeta_K(\sigma)|\zeta_K(s)|\geq 1
\end{equation}
as desired. 
\end{proof}

\begin{cor}For $\sigma>1$,
\[1<\zeta_K(\sigma).\]
\end{cor}
\begin{proof} This follows directly from setting $s=\sigma$ in \eqref{eq:ssigma}.
\end{proof}

Lemma \ref{lem:sandwichzetas} shows that if we can control $\zeta_K(\sigma)$, then we can also control $|\zeta_K(s)|$. We now focus on estimating $\zeta_K(\sigma)$.
\begin{lem}\label{lem:estimatinga}
Let $a_\ell$ be defined by equation  \eqref{eq:a_ddefn}. Then
\[a_\ell\leq \frac{q^\ell}{\ell} + q^{\ell/3}+ \frac{2g}{\ell} (q^{\ell/2}+q^{\ell/4}).\]
\end{lem}
\begin{proof}
This proof follows from combining various elements from \cite{Rosen}*{Theorem 5.12}. By applying M\"obius inversion to \eqref{eq:sumofpi}, we have, for $\ell>1$, 
\begin{equation}\label{eq:lal}
\ell a_\ell =\sum_{d\mid \ell}\mu(d) q^{\ell/d} +\sum_{d\mid \ell}\mu(d)\left(\sum_{j=1}^{2g} \pi_j^{\ell/d}\right).
\end{equation}
We focus on the first term. The highest power of $q$ is $q^\ell$ and the second highest power is  $q^{\ell/2}$ that only appears when $2\mid \ell$ and in that case, it has coefficient $-1$. All the other powers are at most $q^{\ell/3}$. The total number of terms is bounded by 
$\sum_{d\mid \ell}|\mu(d)|$, which is seen to be $2^{\omega(\ell)}$, where $\omega(\ell)$ is the number of distinct prime divisors of $\ell$. If $p_1,\dots,p_{\omega(\ell)}$ are the distinct  primes dividing $\ell$, then one has that $2^{\omega(\ell)}\leq p_1\cdots p_{\omega(\ell)}\leq \ell$. Combining all of this, we obtain 
\begin{equation}\label{eq:firstmu}
\sum_{d\mid \ell}\mu(d) q^{\ell/d}\leq q^\ell + \ell q^{\ell/3}.
\end{equation}
Similarly, we have 
\begin{equation}\label{eq:secondmu}
\left|\sum_{d\mid \ell}\mu(d)\left(\sum_{j=1}^{2g} \pi_j^{\ell/d}\right)\right|\leq 2g q^{\ell/2}+2g q^{\ell/4}.
\end{equation}
The result follows by combining equations \eqref{eq:firstmu} and \eqref{eq:secondmu}.
\end{proof}
We use the previous estimate to give an upper bound for $\zeta_K(\sigma)$. 

\begin{prop}\label{prop:sigma}
Let $\sigma>1$. Then we have 
\[\zeta_K(\sigma)\leq \left(\frac{\exp\left(\frac{1}{q^{\sigma-\frac{1}{3}}-1}\right)}{(1-q^{1-\sigma})(1-q^{\frac{1}{2}-\sigma})^{2g}(1-q^{\frac{1}{4}-\sigma})^{2g}}\right)^{\frac{q^{\sigma}}{q^\sigma-1}}.\]
\end{prop}
\begin{proof}
We apply Lemma \ref{lem:estimatinga} to the Euler product \eqref{eq:Eulera} and obtain
\begin{align*}
\log \zeta_K(\sigma) =&-\sum_{n=1}^\infty a_n\log \left(1-\frac{1}{q^{n\sigma}}\right)\\
\leq &-\sum_{n=1}^\infty \left(\frac{q^n}{n} + q^{n/3}+ \frac{2g}{n} (q^{n/2}+q^{n/4})\right)\log \left(1-\frac{1}{q^{n\sigma}}\right).
\end{align*}
By using the estimate $1-\frac{1}{x}\leq \log(x)$ for $x \in \R_{\geq 0}$, we obtain
\begin{align*}
\log \zeta_K(\sigma)
\leq &\sum_{n=1}^\infty \left(\frac{q^n}{n} + q^{n/3}+ \frac{2g}{n} (q^{n/2}+q^{n/4})\right)\frac{1}{q^{n\sigma}-1}.
\end{align*}
Now we further use the bound 
\[\frac{q^{n\sigma}}{q^{n\sigma}-1}\leq \frac{q^{\sigma}}{q^\sigma-1}\]
and get
\begin{align*}
\log \zeta_K(\sigma)
\leq  &\frac{q^{\sigma}}{q^\sigma-1}\sum_{n=1}^\infty \left(\frac{q^n}{n} + q^{n/3}+ \frac{2g}{n} (q^{n/2}+q^{n/4})\right)q^{-n\sigma}\\
=& \frac{q^{\sigma}}{q^\sigma-1} \left(-\log(1-q^{1-\sigma})-2g\log(1-q^{\frac{1}{2}-\sigma})-2g \log(1-q^{\frac{1}{4}-\sigma})+\frac{1}{q^{\sigma-\frac{1}{3}}-1} \right),
\end{align*}
and thus
\[\zeta_K(\sigma)\leq \left(\frac{\exp\left(\frac{1}{q^{\sigma-\frac{1}{3}}-1}\right)}{(1-q^{1-\sigma})(1-q^{\frac{1}{2}-\sigma})^{2g}(1-q^{\frac{1}{4}-\sigma})^{2g}}\right)^{\frac{q^{\sigma}}{q^\sigma-1}}.\]
\end{proof}

The result of Proposition \ref{prop:sigma} gives an upper bound for $\zeta_K(\sigma)$ when $\sigma>1$. This upper bound tends to infinity as  $g \rightarrow \infty$, and therefore it gives a weak result in terms of the Northcott property. 

We close this section by focusing on the case of quadratic fields, where we obtain a better upper bound, independent of $g$. 

\begin{prop} \label{prop:hyperelliptic} Let $\sigma>1$ and $K$ be a quadratic extension of $\F_q(T)$ with constant field $\F_q$. Then 
\begin{equation*}
|\zeta_K(\sigma)|\leq \frac{1}{(1-q^{-\sigma})(1-q^{1-\sigma})^2}.\end{equation*}
\end{prop}
\begin{proof}
Since $K$ is quadratic, we can write 
\begin{equation}\label{eq:quadratic}
\zeta_K(s)=\frac{L(s,\chi_D)}{(1-q^{-s})(1-q^{1-s})},
\end{equation}
where $\chi_D$ is the quadratic character associated to the extension and 
\[L(s,\chi_D)=\sum_{f \text{monic}} \frac{\chi_D(f)}{|f|^s}.\]

To be concrete, we can think of $\chi_D(f):=\left(\frac{D}{f}\right)_2$, the Legendre symbol, with $D\in \mathcal{H}_{2g+1}$, the set of monic square-free polynomials of degree $2g+1$. Furthermore, we can think of $K=\F_q(T)(\sqrt{D})$.

For $\sigma>1$ we can write, analogously to \eqref{eq:logzeta}, 
\[\log L(\sigma,\chi)=\sum_{f \text{monic}} \frac{\Lambda(f)\chi(f)}{\deg(f) |f|^\sigma},\]
where $\Lambda(f)$ is the von Mangoldt function. This gives
\[|\log L(\sigma,\chi)|\leq \sum_{f \text{monic}} \frac{\Lambda(f)}{\deg(f) |f|^\sigma}=\log \zeta_{\F_q[T]}(\sigma)=-\log(1-q^{1-\sigma}).\]
Now 
\[\log|L(\sigma,\chi)| = \re ( \log L(\sigma,\chi)) \leq |\log L(\sigma,\chi)|\leq -\log(1-q^{1-\sigma}),\]
and thus,
\[|L(\sigma,\chi)| \leq \frac{1}{1-q^{1-\sigma}}.\]
Considering the denominator of $\zeta_K(s)$ in \eqref{eq:quadratic}, we obtain the result. 

\end{proof}

\begin{thm} \label{thm:boundright}
Let $\sigma>1$, and 
\[B\geq \frac{1}{(1-q^{-\sigma})(1-q^{1-\sigma})^2}.\]
Then $(q,s,B)$ 
does not satisfy the Northcott property. 
\end{thm}

\begin{proof}
By Proposition \ref{prop:hyperelliptic},  $|\zeta_K(\sigma)|\leq B$ for all quadratic fields provided that $B$ is larger than $\frac{1}{(1-q^{-\sigma})(1-q^{1-\sigma})^2}$. This gives an infinite family of quadratic fields with $|\zeta_K(\sigma)|\leq B$.

For complex $s$, we use the upper bound in Lemma \ref{lem:sandwichzetas} to conclude.  
\end{proof}

\section{Inside the critical strip} \label{sec:critical}
In this section we use results of Lumley \cites{Lumley,Lumley2}, Li \cite{Li}, and 
Bui, Florea, and Keating \cite{Bui-Florea-Keating} to get information on some specific values inside the critical strip. Unless otherwise stated, we assume that $q\equiv 1 \pmod{4}$. This is a common assumption made in  \cites{Lumley,Lumley2,Andrade-Keating-conj} that allows cleaner formulas as quadratic reciprocity becomes trivial.

\subsection{The Northcott property at the pole $s=1$}

Here we treat the case of $s=1$, corresponding to the right bold point in red of Figure \ref{fig:covering}. We consider the following result of Lumley.
\begin{thm} \cite{Lumley}*{Corollary 1.8} For $g$ large and $1\leq \tau \leq \log g-2\log(\log g )-\log(\log (\log g))$, the number of $D\in \mathcal{H}_{2g+1}$ such that
\[\frac{h_D}{q^g}<\frac{\zeta_{\F_q[T]}(2)}{e^\gamma \tau}\]
is given by 
\begin{equation}\label{eq:Allysa}
(\#\mathcal{H}_{2g+1})\exp\left (-C_1(q^{\{\log\kappa(\tau)\}})\frac{q^{\tau-C_0(q^{\{\log \kappa(\tau)\}})}}{\tau}\left(1+O\left(\frac{\log \tau}{\tau} \right)\right)\right).
\end{equation}
\end{thm}
Above, we have written $h_D$ instead of $h_{\F_q(T)(\sqrt{D})}$ for short, and $\gamma$ denotes the Euler–-Mascheroni constant. We will not discuss $\kappa$, $C_0$, and $C_1$. It suffices to say that $C_0(q^{\{\log\kappa(\tau)\}})$ and $C_1(q^{\{\log\kappa(\tau)\}})$ are positive functions depending on $\tau$. 

\begin{thm} Let $B>0$. Then  $(q,1,B)$ does not satisfy the Northcott property.
\end{thm}
\begin{proof}
Given $B>0$, fix  $\tau$ large enough such that \[\frac{\zeta_{\F_q[T]}(2)}{e^\gamma \tau}\cdot \frac{q}{(1-q^{-1})\log q}<B.\] Since $\tau$ is fixed, we can evaluate the exponential factor in \eqref{eq:Allysa} and it gives a fixed positive constant $c(\tau)$ (that can be very small). 

We have that (see for example, \cite{Rosen}*{Proposition 2.3})\begin{equation} \label{eq:H}
\#\mathcal{H}_{n}=\begin{cases}q^{n} (1-q^{-1}) & n \geq 2,\\
q^n & n=0,1.
\end{cases}
\end{equation}

Applying this, we get that for $g$ large enough (so that $\tau$ satisfies the right conditions)  there are at least 
\[q^{2g+1} (1-q^{-1}) c(\tau)\]
possible $D \in \mathcal{H}_{2g+1}$ satisfying that
\[\frac{h_Dq^{-g}}{(1-q^{-1})\log q}<B.\]
We will combine this with the fact that 
\[\zeta_K^*(1)=\lim_{s\rightarrow 1} \frac{(s-1)L_K(q^{-s})}{(1-q^{-s})(1-q^{1-s})}=\frac{L_K(q^{-1})}{1-q^{-1}}\lim_{s\rightarrow 0}\frac{s-1}{1-q^{1-s}}=\frac{L_K(q^{-1})}{(1-q^{-1})\log q}=\frac{h_Kq^{-g}}{(1-q^{-1})\log q},\]
where the equality $L_K(q^{-1})=h_kq^{-g}$ follows from the functional equation \eqref{functeq}. 

Finally, we obtain
\[\zeta_{K}^*(1)<B.\]

Letting $g \rightarrow \infty$, we have that $S_{q, 1,B}$ is infinite for any choice of $B>0$. 
\end{proof}

\subsection{The segment of the real line  inside the right side of the critical line}
For $1/2<\sigma <1$ we use another result of Lumley that is very similar to the result we had at the pole $s=1$. This corresponds to the red segment in Figure \ref{fig:covering}.

\begin{thm}\cite{Lumley2}*{Theorem 1.3, partial statement}
Let $N$ be large and $1/2 < \sigma < 1$ be fixed. There exist a constant $\beta_q(\sigma)>0$ and an irreducible polynomial $P$ of degree $N$, such that
\begin{align*}
     \log(L(\sigma,\chi_P)) \leq -\beta_q(\sigma)\frac{(\log_q |P|)^{1-\sigma}}{(\log_q\log_q |P|)^\sigma}. 
\end{align*}
\end{thm}

With this, we can prove the following. 
\begin{thm} Let $B>0$ and $1/2<\sigma<1$. Then  $(q,\sigma,B)$ does not satisfy the Northcott property.
\end{thm}
\begin{proof}
Given $B>0$ and $\sigma \in (\frac{1}{2},1)$ we can choose $N$ such that

\begin{align*}
    |\zeta_{K_P}(\sigma)| \leq \frac{1}{|(1-q^{-\sigma})(1-q^{1-\sigma})|}e^{-\beta_q(\sigma)\frac{(\log_q |P|)^{1-\sigma}}{(\log_q\log_q |P|)^\sigma}} \leq B.
\end{align*}
Then, we can construct a sequence of irreducible polynomials $P_k$ where $P_0$ is a polynomial of degree $N$ and $P_k$ is of degree $N+k$ such that for all the polynomials in the sequence 
$$ |\zeta_{K_{P_k}}(\sigma)| \leq B.$$

Thus, we see that $S_{q,\sigma, B}$ is infinite for any choice of $B>0$. 

\end{proof}

\subsection{The Northcott property at $s=1/2$}
Now we consider the case of $s=1/2$, more precisely, we look at $\zeta_K(1/2)$. For this case, we use the following result of Li. 
\begin{thm}\cite{Li}*{Theorem 1.3, simplified version} \label{thm:wanlin} For any $\varepsilon>0$ there exist nonzero constants $B_\varepsilon$ and $N_\varepsilon$ such that if $N>N_\varepsilon$,
\[\#\{ D \in \F_q[T]: D \, \text{monic},\, \text{square-free}, |D|<N, {\textstyle L(\frac{1}{2}, \chi_D)}=0\} \geq B_\varepsilon N^{1/5-\varepsilon}.\]
\end{thm}
The above result immediately implies the following. 
\begin{thm}\label{thm:1/2} Let $B>0$. Then  $(q,1/2,B)$ does not satisfy the Northcott property for $\zeta_K(1/2)$.
\end{thm}
\begin{proof}
By Theorem \ref{thm:wanlin}, there are infinitely many $K$ for which $|\zeta_K(1/2)|=0$ and therefore we obtain infinitely many $K$ such that $|\zeta_K(1/2)|<B$.
\end{proof}

\begin{rem} The above result does not cover the case of $\zeta_K^*(1/2)$. To do this, we would have to consider the first nonzero coefficient of the Taylor series for $\zeta_K(s)$ around $s=1/2$. 
\end{rem}

Theorem \ref{thm:1/2}  can be extended to a more general class of $s$. An algebraic integer $\alpha$ is called a \textbf{Weil integer} if $|\alpha|=\sqrt{q}$ under every complex embedding.

\begin{thm}
For any $B>0$ and $q^s$ a Weil integer, the triple $(q,s,B)$ does not satisfy the Northcott property for $\zeta_K(1/2)$.
\end{thm}

\begin{proof}
The statement follows from the fact that for any Weil integer $q^s$, there exist infinitely many function fields $K$ such that $\zeta_{K}(s)=0$.

By the theory of Honda--Tate, for every Weil integer $q^s$, there exists an abelian variety $A/\mathbb{F}_q$ such that $q^s$ is a Frobenius eigenvalue for $A$. By the work of Gabber \cite{Gabber}*{Corollay 2.5}, for any abelian variety $A/\mathbb{F}_q$, there exists a smooth projective curve $C/\mathbb{F}_q$ such that $A$ is an isogeny factor of the Jacobian of $C$ (see also \cite{BL} for an effective statement). Hence $\zeta_{K_C}(s)=0$ where $K_C$ is the function field of $C$. The theorem follows from the fact that $\zeta_L(s)=0$ for any field $L$ which is an extension of $K$ with constant field $\mathbb{F}_q$.
\end{proof}

\subsection{The Northcott property in the right of the critical line}
Here we consider the set $1/2< \re(s)< 1$. This case will be studied in the context of  the Shifted Moments Conjecture, formulated over function fields by Andrade and Keating \cite{Andrade-Keating-conj} and recently proven for products of up to three factors and $\re(s)<1$ by Bui, Florea, and Keating \cite{Bui-Florea-Keating}.

For simplicity of notation we will write $\zeta_q(s)$ instead of $\zeta_{\F_q[T]}(s)$. Here, as before, it is assumed that   $q\equiv 1 \pmod{4}$ for simplicity. 

We start by recalling a simplified version of one of the results of Bui, Florea, and Keating. 
\begin{thm}\label{thm:BFK}\cite{Bui-Florea-Keating}*{Theorem 1.2, simplified version of a particular case}
Let $\alpha_1, \alpha_2 \in \C$ such that $|\re(\alpha_j)|<1/2$. Denote $A=\{\alpha_1,\alpha_2\}$. For a set $\mathfrak{A}\subseteq A$, let $\mathfrak{A}^-=\{-\mathfrak{a}\, :\, \mathfrak{a} \in \mathfrak{A}\}$ and $q^{-2g \mathfrak{A}}=q^{-2g\sum_{\mathfrak{a} \in \mathfrak{A}} \beta}$.  We have
\[\frac{1}{\# \mathcal{H}_{2g+1}} \sum_{D \in \mathcal{H}_{2g+1}}  {\textstyle L(\frac{1}{2}+\alpha_1, \chi_D)}{\textstyle L(\frac{1}{2}+\alpha_2, \chi_D)}=\sum_{\mathfrak{A}\subseteq A} q^{-2g \mathfrak{A}} S_{(A \setminus \mathfrak{A})\cup \mathfrak{A}^-} +E_2,\]
 where if $C=\{\gamma_1,\gamma_2\}$, 
 \[S_C=\mathcal{A}_C(1)\prod_{1\leq i \leq j \leq 2}\zeta_q(1+\gamma_i+\gamma_j),\]
 \[\mathcal{A}_C(u)= \prod_{\substack{P\, \text{monic}\\ \text{irreducible}}} \prod_{1\leq i \leq j \leq 2}\left(1-\frac{u^{2\deg(P)}}{|P|^{1+\gamma_i+\gamma_j}}\right)\left(1+\left(1+\frac{1}{|P|}\right)^{-1}\sum_{\ell=1}^\infty \frac{\tau_C(P^{2\ell})}{|P|^\ell}u^{2\ell\deg(P)}\right),\]
\begin{equation}\label{defn:tau}
\tau_C(f)=\sum_{\substack{f=f_1f_2\\f_i \text{monic}}} \frac{1}{|f_1|^{\gamma_1}|f_2|^{\gamma_2}}
\end{equation}
 and, for $\varepsilon>0$,
 \[E_2\ll_\varepsilon q^{-(1+2\min\{|\re(\alpha_1)|, |\re(\alpha_2)|\})g+\varepsilon g}.\]
\end{thm}
Remark that the various sets of the form $(A \setminus \mathfrak{A})\cup \mathfrak{A}^-$ should be taken as multi-sets in the case where a parameter is repeated. Therefore, in our case, they always have cardinality 2.

We stress that the result of Bui, Florea, and Keating is much more general than Theorem \ref{thm:BFK}, as it considers a product of $k$ factors of the form ${\textstyle L(\frac{1}{2}+\alpha, \chi_D)}$ and it includes a twist by $\chi_D(h)$, where $h$ is a polynomial of degree very small compared to $g$. We have written a simplified version that is sufficient for our purposes. The error term that we give in Theorem \ref{thm:BFK} is more detailed than the term in the original statement in \cite{Bui-Florea-Keating} and has been taken from the proof.

\begin{thm} \label{thm:BLKconsequence}
 Let $0<\re(\alpha)<\frac{1}{2}$  and
\begin{align*}
B>&\left|\frac{1}{\left(1-q^{-\frac{1}{2}-\alpha}\right)\left(1-q^{\frac{1}{2}-\alpha}\right)}\right|\\&\times \prod_{\substack{P\, \text{monic}\\ \text{irreducible}}} \left[\frac{1}{2}\left( \left(1-\frac{1}{|P|^{\frac{1}{2}+\alpha}}\right)^{-1}\left(1-\frac{1}{|P|^{\frac{1}{2}+\overline{\alpha}}}\right)^{-1}+\left(1+\frac{1}{|P|^{\frac{1}{2}+\alpha}}\right)^{-1}\left(1+\frac{1}{|P|^{\frac{1}{2}+\overline{\alpha}}}\right)^{-1}\right)+\frac{1}{|P|}\right]^{1/2}\\
 &\times \left(1+\frac{1}{|P|}\right)^{-1/2}.
\end{align*}
Then $(q,1/2+\alpha,B)$ does not satisfy the Northcott property. 
 
\end{thm}

\begin{proof}
 In Theorem  \ref{thm:BFK} we fix $\alpha=\alpha_1, \alpha_2=\overline{\alpha}$.
  Since $\re(\alpha)>0$,  the dominant term occurs when $\mathfrak{A}$ is the empty set, and we have  \begin{align}\label{eq:sum2}&\frac{1}{\# \mathcal{H}_{2g+1}} \sum_{D \in \mathcal{H}_{2g+1}}  |{\textstyle L(\frac{1}{2}+\alpha, \chi_D)}|^2=   \prod_{\substack{P\, \text{monic}\\ \text{irreducible}}} \left(1+\left(1+\frac{1}{|P|}\right)^{-1}\sum_{\ell=1}^\infty \frac{\tau_{\{\alpha,\overline{\alpha}\}}(P^{2\ell})}{|P|^\ell}\right)(1+o(1)).
 \end{align}
We can give a more precise expression for the Euler product above. Notice that
\begin{align*}
& \frac{1}{2}\left( \left(1-\frac{1}{|P|^{\frac{1}{2}+\alpha}}\right)^{-1}\left(1-\frac{1}{|P|^{\frac{1}{2}+\overline{\alpha}}}\right)^{-1}+\left(1+\frac{1}{|P|^{\frac{1}{2}+\alpha}}\right)^{-1}\left(1+\frac{1}{|P|^{\frac{1}{2}+\overline{\alpha}}}\right)^{-1}\right)\\
&=\frac{1}{2}\left( \sum_{j_1=0}^\infty \frac{1}{|P|^{j_1(\frac{1}{2}+\alpha)}} \sum_{j_2=0}^\infty \frac{1}{|P|^{j_2(\frac{1}{2}+\overline{\alpha})}}+\sum_{j_1=0}^\infty \frac{(-1)^{j_1}}{|P|^{j_1(\frac{1}{2}+\alpha)}} \sum_{j_2=0}^\infty \frac{(-1)^{j_2}}{|P|^{j_2(\frac{1}{2}+\overline{\alpha})}}\right) \\
&=1+\frac{1}{2}\left(\sum_{\ell=1}^\infty\frac{ \tau_{\{\alpha,\overline{\alpha}\}}(P^{\ell})}{|P|^\frac{\ell}{2}} +\sum_{\ell=1}^\infty\frac{(-1)^\ell \tau_{\{\alpha,\overline{\alpha}\}}(P^{\ell})}{|P|^\frac{\ell}{2}} \right).
\end{align*}

 This allows us to write 
   \[\frac{1}{\# \mathcal{H}_{2g+1}} \sum_{D \in \mathcal{H}_{2g+1}}  |{\textstyle L(\frac{1}{2}+\alpha, \chi_D)}|^2= C_\alpha(1+o(1)),\]
   where 
 \begin{align*}
 C_\alpha=&\prod_{\substack{P\, \text{monic}\\ \text{irreducible}}}
 \left[\frac{1}{2}\left( \left(1-\frac{1}{|P|^{\frac{1}{2}+\alpha}}\right)^{-1}\left(1-\frac{1}{|P|^{\frac{1}{2}+\overline{\alpha}}}\right)^{-1}+\left(1+\frac{1}{|P|^{\frac{1}{2}+\alpha}}\right)^{-1}\left(1+\frac{1}{|P|^{\frac{1}{2}+\overline{\alpha}}}\right)^{-1}\right)+\frac{1}{|P|}\right]\\ &\times \left(1+\frac{1}{|P|}\right)^{-1}.
 \end{align*}
 In other words, the average value of  $|{\textstyle L(\frac{1}{2}+\alpha, \chi_D)}|^2$ is given by $C_\alpha (1+o(1))$.

Finally, given $\varepsilon>0$, we can guarantee  that for $g$ large enough, there is a $D\in \mathcal{H}_{2g+1}$ such that 
\[|{\textstyle L(\frac{1}{2}+\alpha, \chi_D)}|\leq C_\alpha^{1/2}+\varepsilon.\]
Taking $\varepsilon$ arbitrarily small we can construct an infinite sequence of $D$'s satisfying this property, and leading to bounded $|\zeta_K(1/2+\alpha)|$. This implies that $S_{q,1/2+\alpha,B}$ is infinite. 

 \end{proof}

\subsection{The Northcott property at $\re(s)=1$.}

In this section we examine the behaviour at the boundary of the critical strip, namely $\re(s)=1$. We need a result along the lines of Theorem \ref{thm:BFK}. Since we have not found this in the literature, we start by computing the average of $|{\textstyle L(\frac{1}{2}+\alpha,\chi_D)}|^2$ where $\re(\alpha)\geq \frac{1}{2}$. When $\re(\alpha)>\frac{1}{2}$ this also gives an alternative better bound for Theorem \ref{thm:boundright}.

We assume that $q \equiv 1 \pmod{4}$. In particular, this implies that quadratic reciprocity is trivial and that $q>4$. In this section we will use  $\mathcal{M}$ to denote the monic polynomials in $\F_q[T]$, $\mathcal{M}_n$ to denote the monic polynomials of degree $n$, and $\mathcal{M}_{\leq n}$ to denote those of degree up to $n$. 

Our main result here is the following.
\begin{thm}\label{thm:averageline} Let $\re(\alpha)\geq \frac{1}{2}$. Then  for $\varepsilon>0$, \[\frac{1}{\# \mathcal{H}_{2g+1}} \sum_{D \in \mathcal{H}_{2g+1}}  |{\textstyle L(\frac{1}{2}+\alpha,\chi_D)}|^2=\prod_{\substack{P \text{monic}\\ \text{irreducible}}} \left(1+ \left( 1+\frac{1}{|P|}\right)^{-1}  \sum_{\ell=1}^\infty  \frac{\tau_{\{\alpha,\overline{\alpha}\}}(P^{2\ell})}{|P|^\ell} \right)+O(q^{(\varepsilon-\re(\alpha))g}+ 4^g q^{(\varepsilon-2\re(\alpha))g}),\]
where for a set $C$, $\tau_C$ is defined by \eqref{defn:tau}.
\end{thm}
We remark that the main term in the above formula is the same as in equation \eqref{eq:sum2}, with a difference in the conditions, namely that we now have $\re(\alpha)\geq \frac{1}{2}$. Before proceeding to the proof of Theorem \ref{thm:averageline}, we consider some auxiliary results. 

\begin{lem} \label{lem:approxfunceq} \cite{Bui-Florea-Keating}*{Lemma 2.1, particular case} 
We have
 \[|{\textstyle L(\frac{1}{2}+\alpha,\chi_D)}|^2=\sum_{f \in \mathcal{M}_{\leq 2g}} \frac{\tau_{\{\alpha,\overline{\alpha}\}}(f) \chi_D(f)}{|f|^\frac{1}{2}}+q^{-4g\re(\alpha)}\sum_{f \in \mathcal{M}_{\leq 2g-1}} \frac{\tau_{\{-\alpha,-\overline{\alpha}\}}(f) \chi_D(f)}{|f|^\frac{1}{2}},\]
 where for a set $C$, $\tau_C$ is defined by \eqref{defn:tau}.
\end{lem}
\begin{proof}
 This follows from \cite{Bui-Florea-Keating}*{Lemma 2.1} by setting $k=2$ and $\alpha_1=\alpha$, $\alpha_2=\overline{\alpha}$.
\end{proof}

\begin{lem} \cite{Faifman-Rudnick}*{Lemma 2.1} \label{lem:FR}
Let $\chi_f$ be a non-trivial Dirichlet character modulo $f$. Then for $n<\deg(f)$,
\[\left|\sum_{B\in \mathcal{M}_n} \chi_f(B)\right| \leq \binom{\deg(f)-1}{n} q^\frac{n}{2}.\]
\end{lem}

\begin{lem}\cite{Bui-Florea}*{Lemma 3.7}\label{lem:BFsquares} For $f \in \mathcal{M}$ we have 
\[\frac{1}{\# \mathcal{H}_{2g+1}} \sum_{D \in \mathcal{H}_{2g+1}} \chi_D(f^2) = \prod_{\substack{P \text{monic}\\ \text{irreducible}\\P\mid f}} \left(1+\frac{1}{|P|}\right)^{-1}+O(q^{-2g}).\]
\end{lem}

\begin{lem} For any $\varepsilon>0$,
\begin{equation}\label{eq:tau}
|\tau_{\{\alpha,\overline{\alpha}\}}(f)| \ll |f|^{\varepsilon-\re(\alpha)}\end{equation} and similarly 
\begin{equation}\label{eq:tau2}
|\tau_{\{-\alpha,-\overline{\alpha}\}}(f)| \ll |f|^{\varepsilon+\re(\alpha)}.
\end{equation}
\end{lem}
\begin{proof}
 We have that 
 \[|\tau_{\{\alpha,\overline{\alpha}\}}(f)| =\left|\sum_{\substack{f=f_1f_2\\f_i \text{monic}}} \frac{1}{|f_1|^{\alpha}  |f_2|^{\overline{\alpha}}}\right|\leq \frac{d_2(f)}{|f|^{\re(\alpha)}},\]
 where $d_2$ is the divisor function. We can use that $d_2(f)=o(|f|^\varepsilon)$.
 The bound for $\tau_{\{-\alpha,-\overline{\alpha}\}}(f)$ is proven similarly.
\end{proof}

We are now ready to proceed with the proof of the main result of this section. 
\begin{proof}[Proof of Theorem \ref{thm:averageline}]
 By Lemma \ref{lem:approxfunceq}, we can split the sum under consideration in two Dirichlet sums of approximate length $2g$ as follows
 \begin{align*}
  \frac{1}{\# \mathcal{H}_{2g+1}} \sum_{D \in \mathcal{H}_{2g+1}}  |{\textstyle L(\frac{1}{2}+\alpha,\chi_D)}|^2= &  \frac{1}{\# \mathcal{H}_{2g+1}} \sum_{D \in \mathcal{H}_{2g+1}} \sum_{f \in \mathcal{M}_{\leq 2g}} \frac{\tau_{\{\alpha,\overline{\alpha}\}}(f) \chi_D(f)}{|f|^\frac{1}{2}}\\&+ \frac{q^{-2g}}{\# \mathcal{H}_{2g+1}} \sum_{D \in \mathcal{H}_{2g+1}}\sum_{f \in \mathcal{M}_{\leq 2g-1}} \frac{\tau_{\{-\alpha,-\overline{\alpha}\}}(f) \chi_D(f)}{|f|^\frac{1}{2}}\\
  =& S_{2g,\alpha}+S_{2g-1,-\alpha}.
 \end{align*}
We can further split each of the above sums into a sum where the character is evaluated in squares and a sum where it is not as follows 
\begin{align*}
  S_{2g,\alpha} = & \frac{1}{\# \mathcal{H}_{2g+1}} \sum_{D \in \mathcal{H}_{2g+1}} \sum_{\substack{f \in \mathcal{M}_{\leq 2g}\\f=\square}} \frac{\tau_{\{\alpha,\overline{\alpha}\}}(f) \chi_D(f)}{|f|^\frac{1}{2}}+
  \frac{1}{\# \mathcal{H}_{2g+1}} \sum_{D \in \mathcal{H}_{2g+1}} \sum_{\substack{f \in \mathcal{M}_{\leq 2g}\\f\not = \square}} \frac{\tau_{\{\alpha,\overline{\alpha}\}}(f) \chi_D(f)}{|f|^\frac{1}{2}}\\
  =&  S_{2g,\alpha}(\square) +  S_{2g,\alpha} (\not = \square),\\
  S_{2g-1,-\alpha}=& \frac{q^{-4\re(\alpha)g}}{\# \mathcal{H}_{2g+1}} \sum_{D \in \mathcal{H}_{2g+1}}\sum_{\substack{f \in \mathcal{M}_{\leq 2g-1}\\f=\square }} \frac{\tau_{\{-\alpha,-\overline{\alpha}\}}(f) \chi_D(f)}{|f|^\frac{1}{2}}+\frac{q^{-4\re(\alpha)g}}{\# \mathcal{H}_{2g+1}} \sum_{D \in \mathcal{H}_{2g+1}}\sum_{\substack{f \in \mathcal{M}_{\leq 2g-1}\\f \not = \square}} \frac{\tau_{\{-\alpha,-\overline{\alpha}\}}(f) \chi_D(f)}{|f|^\frac{1}{2}}\\
   =&  S_{2g-1,-\alpha}(\square) +  S_{2g-1,-\alpha} (\not = \square).
\end{align*}
The main term comes from $S_{2g,\alpha}(\square)$, while the other three terms are smaller. We do not need to estimate them for our purposes, so we will bound them to get an error term.

We start by focusing on the main term, coming from $S_{2g,\alpha}(\square)$. We apply Lemma \ref{lem:BFsquares} and obtain 
\begin{align*}
\frac{1}{\# \mathcal{H}_{2g+1}}\sum_{D \in \mathcal{H}_{2g+1}} \sum_{\substack{f \in \mathcal{M}_{\leq 2g}\\f=\square}} \frac{\tau_{\{\alpha,\overline{\alpha}\}}(f) \chi_D(f)}{|f|^\frac{1}{2}} = & \sum_{\substack{h \in \mathcal{M}_{\leq g}}} \frac{\tau_{\{\alpha,\overline{\alpha}\}}(h^2)}{|h|} \frac{1}{\# \mathcal{H}_{2g+1}}\sum_{D \in \mathcal{H}_{2g+1}}\chi_D(h^2)\\
= & \sum_{\substack{h \in \mathcal{M}_{\leq g}}} \frac{\tau_{\{\alpha,\overline{\alpha}\}}(h^2)}{|h|}  \left(\prod_{P\mid h} \left( 1+\frac{1}{|P|}\right)^{-1} + O(q^{-2g})\right).
\end{align*}
We remark that there are $q^n$ monic polynomials of degree $n$. Applying 
inequality \eqref{eq:tau}, we obtain
\[\sum_{\substack{h \in \mathcal{M}_{\leq g}}} \frac{|\tau_{\{\alpha,\overline{\alpha}\}}(h^2)|}{|h|} O (q^{-2g})\leq \sum_{n=1}^g q^{2n(\varepsilon-\re(\alpha))} O (q^{-2g})
\ll O(q^{(\varepsilon-2\re(\alpha)-2)g}).\]

 Now we consider the generating function
 \begin{align*}
  \mathcal{A}(u)=& \sum_{\substack{h \in \mathcal{M}}} \frac{\tau_{\{\alpha,\overline{\alpha}\}}(h^2)}{|h|}  \prod_{P\mid h} \left( 1+\frac{1}{|P|}\right)^{-1} u^{\deg(h)}
  = \prod_{\substack{P \text{monic}\\\text{irreducible}}} \left(1+ \left( 1+\frac{1}{|P|}\right)^{-1}  \sum_{\ell=1}^\infty  \frac{\tau_{\{\alpha,\overline{\alpha}\}}(P^{2\ell})}{|P|^\ell} u^{2\ell\deg(P)}\right).
 \end{align*}
Since $|\tau_{\{\alpha,\overline{\alpha}\}}(P^{2\ell})|\leq \frac{2\ell+1}{|P|^{2\ell\re(\alpha)}}$, it can be seen that $\mathcal{A}(u)$ converges for $|u|<q^{\re(\alpha)}$. 
 
 By Perron's formula, for $r<1$,
 \begin{align*}
  \sum_{\substack{h \in \mathcal{M}_{\leq g}}} \frac{\tau_{\{\alpha,\overline{\alpha}\}}(h^2)}{|h|}  \prod_{P\mid h} \left( 1+\frac{1}{|P|}\right)^{-1} =&\frac{1}{2\pi i} \oint_{|u|=r}  \frac{\mathcal{A}(u)}{u^g(1-u)}\frac{du}{u}.
 \end{align*}
We move the integral to the circle $|u|=q^{\re(\alpha)-\varepsilon}$ encountering the pole at $u=1$. This gives
\begin{align*}
   \sum_{\substack{h \in \mathcal{M}_{\leq g}}} \frac{\tau_{\{\alpha,\overline{\alpha}\}}(h^2)}{|h|}  \prod_{P\mid h} \left( 1+\frac{1}{|P|}\right)^{-1} =&-\Res_{u=1}  \frac{\mathcal{A}(u)}{u^{g+1}(1-u)}+O(q^{(\varepsilon-\re(\alpha))g})=\mathcal{A}(1)+O(q^{(\varepsilon-\re(\alpha))g}).
\end{align*}
Putting all of the above together, we finally write 
\begin{equation}\label{eq:squareg}
S_{2g,\alpha}(\square)=\prod_{\substack{P \text{monic}\\\text{irreducible}}} \left(1+ \left( 1+\frac{1}{|P|}\right)^{-1}  \sum_{\ell=1}^\infty  \frac{\tau_{\{\alpha,\overline{\alpha}\}}(P^{2\ell})}{|P|^\ell} \right)+O(q^{(\varepsilon-\re(\alpha))g}).
\end{equation}
 
Now we consider $S_{2g-1,-\alpha}(\square)$. Following similar steps as before and applying \eqref{eq:tau2},
\begin{align*}
\left|\frac{1}{\# \mathcal{H}_{2g+1}}\sum_{D \in \mathcal{H}_{2g+1}} \sum_{\substack{f \in \mathcal{M}_{\leq 2g-1}\\f=\square}} \frac{\tau_{\{-\alpha,-\overline{\alpha}\}}(f) \chi_D(f)}{|f|^\frac{1}{2}} \right|
= & \left|\sum_{\substack{h \in \mathcal{M}_{\leq g-1}}} \frac{\tau_{\{-\alpha,-\overline{\alpha}\}}(h^2)}{|h|}  \prod_{P\mid h} \left( 1+\frac{1}{|P|}\right)^{-1} \right|+ O(q^{-2g})\\
\ll  &\sum_{\substack{h \in \mathcal{M}_{\leq g-1}}}|h|^{2\re(\alpha)-1+\varepsilon} + O(q^{-2g})\\
\ll & q^{(2\re(\alpha)+\varepsilon)g}.
\end{align*}
Combining with the term $q^{-4\re(\alpha)g}$, this gives
\begin{equation}\label{eq:squareg1}
S_{2g-1,-\alpha}(\square) \ll q^{(\varepsilon-2\re(\alpha))g}.
\end{equation}

Our next step is to bound $S_{2g,\alpha}(\not = \square)$. We follow the proof of \cite{Andrade}*{Lemma 3}. 
\begin{align*}
\sum_{D \in \mathcal{H}_{2g+1}} \sum_{\substack{f \in \mathcal{M}_{\leq 2g}\\f\not = \square}} \frac{\tau_{\{\alpha,\overline{\alpha}\}}(f) \chi_D(f)}{|f|^\frac{1}{2}}
 = &\sum_{n=0}^{2g} q^{-\frac{n}{2}} \sum_{\substack{f\in \mathcal{M}_n\\f\not =\square}} \tau_{\{\alpha,\overline{\alpha}\}}(f)  \sum_{D \in \mathcal{H}_{2g+1}}  \chi_D(f)\\
= & \sum_{n=0}^{2g} q^{-\frac{n}{2}} \sum_{\substack{f\in \mathcal{M}_n\\f\not =\square}} \tau_{\{\alpha,\overline{\alpha}\}}(f)  \sum_{D\in \mathcal{M}_{2g+1}} \sum_{\substack{A \in \mathcal{M}\\ A^2\mid D}}\mu(A)\chi_D(f)\\ 
=& \sum_{n=0}^{2g} q^{-\frac{n}{2}} \sum_{\substack{f\in \mathcal{M}_n\\f\not =\square}} \tau_{\{\alpha,\overline{\alpha}\}}(f) \sum_{\substack{A \in \mathcal{M}_{\leq g}}}\mu(A) \sum_{B \in \mathcal{M}_{2g+1-2\deg(A)}}\chi_f(B),
\end{align*}
where in the last line we have applied quadratic reciprocity, which is trivial for $q\equiv 1 \pmod{4}$.

We can apply Lemma \ref{lem:FR} in the innermost sum when $2g+1-2\deg(A)<\deg(f)$. Note that the sum is zero otherwise, since it is a full character sum. Also applying \eqref{eq:tau}, we obtain 
\begin{align*}
\left|\sum_{D \in \mathcal{H}_{2g+1}} \sum_{\substack{f \in \mathcal{M}_{\leq 2g}\\f\not = \square}} \frac{\tau_{\{\alpha,\overline{\alpha}\}}(f) \chi_D(f)}{|f|^\frac{1}{2}}\right|\leq &  \sum_{n=0}^{2g} q^{-\frac{n}{2}} \sum_{\substack{f\in \mathcal{M}_n\\f\not =\square}} |\tau_{\{\alpha,\overline{\alpha}\}}(f)| \sum_{\substack{A \in \mathcal{M}_{\leq g}}} \binom{\deg(f)-1}{2g+1-2\deg(A)} q^\frac{2g+1-2\deg(A)}{2}\\
\ll & q^g \sum_{n=0}^{2g} q^{-\frac{n}{2}} \sum_{\substack{f\in \mathcal{M}_n}} |f|^{\varepsilon-\re(\alpha)}2^{\deg(f)-1}\\
\ll &  q^g \sum_{n=0}^{2g} 2^n q^{(\frac{1}{2}-\re(\alpha)+\varepsilon) n}\\
\ll & 2^{2g} q^{(2-2\re(\alpha)+\varepsilon)g}.
 \end{align*}
Combining with equation \eqref{eq:H}, we get 
\begin{equation}\label{eq:nsquareg}
S_{2g,\alpha}(\not = \square) \ll  2^{2g} q^{(\varepsilon-2\re(\alpha))g}.
\end{equation}

Finally we consider $S_{2g-1,-\alpha}(\not = \square)$. The computation is similar as before until we reach 
 \begin{align*}
\left|\sum_{D \in \mathcal{H}_{2g+1}} \sum_{\substack{f \in \mathcal{M}_{\leq 2g-1}\\f\not = \square}} \frac{\tau_{\{-\alpha,-\overline{\alpha}\}}(f) \chi_D(f)}{|f|^\frac{1}{2}}\right|
 \ll  &   q^g \sum_{n=0}^{2g-1} q^{-\frac{n}{2}} \sum_{\substack{f\in \mathcal{M}_n}} |\tau_{\{-\alpha,-\overline{\alpha}\}}(f)|2^{\deg(f)-1}\\
 \ll &  q^g \sum_{n=0}^{2g-1} q^{-\frac{n}{2}} \sum_{\substack{f\in \mathcal{M}_n}} 2^{\deg(f)} |f|^{\varepsilon+\re(\alpha)}\\
 \ll & 2^{2g} q^{(2+2\re(\alpha)+\varepsilon)g},
   \end{align*}
 where we have used equation \eqref{eq:tau2}.   
Again, combining with the extra factor $\frac{q^{-4\re(\alpha)g}}{\# \mathcal{H}_{2g+1}}$, this leads to 
\begin{equation}\label{eq:nsquareg1}
S_{2g-1,-\alpha}(\not = \square) \ll 2^{2g} q^{(\varepsilon-2\re(\alpha))g}.
\end{equation}
The result follows by combining equations \eqref{eq:squareg}, \eqref{eq:squareg1}, \eqref{eq:nsquareg} and \eqref{eq:nsquareg1}

\end{proof}

With Theorem \ref{thm:averageline} proven, we can now proceed to study the corresponding Northcott property. 

\begin{thm} \label{thm:BLKconsequence1}
 Let $\re(\alpha)\geq \frac{1}{2}$ with $\alpha\not =\frac{1}{2}$ and
\begin{align*}
B>&\left|\frac{1}{\left(1-q^{-\frac{1}{2}-\alpha}\right)\left(1-q^{\frac{1}{2}-\alpha}\right)}\right|\\&\times \prod_{\substack{P\, \text{monic}\\ \text{irreducible}}} \left[\frac{1}{2}\left( \left(1-\frac{1}{|P|^{\frac{1}{2}+\alpha}}\right)^{-1}\left(1-\frac{1}{|P|^{\frac{1}{2}+\overline{\alpha}}}\right)^{-1}+\left(1+\frac{1}{|P|^{\frac{1}{2}+\alpha}}\right)^{-1}\left(1+\frac{1}{|P|^{\frac{1}{2}+\overline{\alpha}}}\right)^{-1}\right)+\frac{1}{|P|}\right]^{1/2}\\
 &\times \left(1+\frac{1}{|P|}\right)^{-1/2}.
\end{align*}
Then $(q,1/2+\alpha,B)$ does not satisfy the Northcott property. 
 \end{thm}
\begin{proof}
The proof follows the same lines as the proof of Theorem \ref{thm:BLKconsequence}.
\end{proof}

\begin{rem}
 As a final remark, for $\re(\alpha)>\frac{1}{2}$, Theorem \ref{thm:BLKconsequence1} provides a better result than Theorem \ref{thm:boundright}. In the case of Theorem \ref{thm:boundright}, the bound is chosen to control all the quadratic $\zeta_K(s)$, while in the case of  Theorem \ref{thm:BLKconsequence1} the bound is chosen according to the average, and is, therefore, less resctricted.
\end{rem}

\end{document}